\documentclass[11pt]{amsart}
  
\usepackage{amsmath}
\usepackage{amstext}
\usepackage{textcomp}
\usepackage{amssymb}
\usepackage{mathrsfs}
\usepackage[english]{babel}
\usepackage{mathtools}
\usepackage{longtable}
\usepackage{float}
\usepackage{array}
\usepackage{varioref}
\usepackage{enumerate}
\usepackage{graphicx}
\newtheorem{theorem}{Theorem}[section]
\newtheorem{lemma}[theorem]{Lemma}

\newtheorem{remark}{Remark}[section]
\newtheorem{conjecture}[theorem]{Conjecture}
\newtheorem*{theorem*}{Theorem}
\newtheorem*{corollary*}{Corollary}
\usepackage{url}
\usepackage[utf8]{inputenc}
\usepackage{multirow}
\usepackage{booktabs}
\usepackage{tikz-cd}
\usepackage{color}

\newcommand{\RN}[1]{
\textup{\uppercase\expandafter{\romannumeral#1}}
}
\newcommand{\N}{\rm{N}}

\date{08-10-2020}

\begin{document}
\title[Size of the Fourier coefficients] 
{On the size of the Fourier coefficients of Hilbert cusp forms}
\author{Balesh Kumar}
\address{Department of Mathematics, Indian Institute of Technology Ropar\\
 Rupnagar-140001\\
 Punjab, INDIA}
\email[Balesh Kumar]{baleshkumar@iitrpr.ac.in; baleshmaths@gmail.com}

\begin{abstract}
Let $\bf f$ be a primitive Hilbert cusp form of weight $k$ and level $\mathfrak{n}$ with Fourier coefficients $c_{\bf f}(\mathfrak{m})$. We prove a non-trivial upper bound for almost all Fourier coefficients $c_{\bf f}(\mathfrak{m})$  of $\bf f$. This generalizes the bounds obtained by Luca, Radziwi\l{}\l{} and Shparlinski. We also prove the existence of infinitely many integral ideals $\mathfrak{m}$ for which the Fourier coefficients $c_{\bf f}(\mathfrak{m})$ have the improved upper bound and further we obtain a refinement of these integral ideals in terms of prime powers. In particular, this enable us to deduce the bound for Fourier coefficients of elliptic cusp forms beyond the `typical size'.
Moreover, we prove further improvements of the bound under the assumption of Littlewood's conjecture. Finally, We study a lower bound for the Fourier coefficients at prime powers provided the corresponding Hecke eigen angle is badly approximable.
\end{abstract}
\maketitle
\section{Introduction}
Let $F$ be a totally real number field of degree $m$ and let $k:=(k_{1},\ldots,k_{m})\in\Bbb Z^{m}$ with $k_{j}>0$, $1\leq j\leq m$. For an integral ideal $\mathfrak{n}$ of $F$, let $\mathfrak{S}_{k}(\mathfrak{n})$ denote the space of all Hilbert cusp forms of weight $k$ and level $\mathfrak{n}$ and $\mathscr{F}_{k}(\mathfrak{n})$ denote the set of all primitive forms of weight $k$ and level $\mathfrak{n}$. Let $\bf f$ be a cusp form in $\mathscr{F}_{k}(\mathfrak{n})$ with  Fourier coefficients ${\N}(\mathfrak{m})^{{(k_{0}-1})/2}c_{\bf f}(\mathfrak{m})$ where $k_{0}:=\max\{k_1,\ldots,k_{m}\}$ and $\mathfrak{m}$ varies over all non-zero integral ideals of $F$. (see section \ref{priliminaries} for the details).

The Fourier coefficients of modular forms are of great interest because of their arithmetic and algebraic information that they encode. One of the classical problem is to study the size of the Fourier coefficients of modular forms. It is well known by a standard argument of Hecke that
$|c_{\bf f}(\mathfrak{m})|\ll{\N}(\mathfrak{m})^{1/2}$. However, by the deep results of Deligne (for $m=1$) and Blasius (for Hilbert cusp forms), we have $|c_{\bf f}(\mathfrak{m})|\leq d(\mathfrak{m})$, where $d(\mathfrak{m})$ is the divisor function of $\mathfrak{m}$.

An interesting question is to investigate the typical size of $c_{\bf f}(\mathfrak{m})$ for almost all $\mathfrak{m}$. Recently, in \cite[Corollary 2]{Luka-Radziwill-Shaperlinski}, Luca, Radziwi\l{}\l{} and Shparlinski obtained a refined version of Deligne's bound for almost all the Fourier coefficients of elliptic cusp forms. They proved the following result: 
\begin{theorem}[Luca, Radziwi\l{}\l{} and Shparlinski] Let $f$ be an elliptic cusp form which is a normalized Hecke eigenform of weight $k$ for the full modular group with Fourier coefficients $c_{f}(n)$. Then
$$
c_{f}(n)\ll (\log n)^{-\frac{1}{2}+\epsilon}
$$
holds for a set of natural numbers $n$ with natural density $1$. 
\end{theorem}
One of the aim of the present article is to generalize the above estimate for Fourier coefficients of Hilbert cusp forms. In the first theorem of the paper, we prove the following:
\begin{theorem}\label{Hilhalfepsilon}
Let $\bf f$ be a primitive cusp form in $\mathscr{F}_{k}(\mathfrak{n})$ of weight $k=(k_{1},\ldots,k_{m})$ such that $k_{1}\equiv\cdots\equiv k_{m}\equiv 0\pmod{2}$ and each $k_{j}\geq 2$. Suppose $\bf f$ does not have complex multiplication and	
let $c_{{\bf f}}(\mathfrak{m})$ be the normalized  coefficients of $\bf f$. Then we have
$$
c_{{\bf f}}(\mathfrak{m})\ll (\log{\N}(\mathfrak{m}))^{-\frac{1}{2}+\epsilon}
$$
for all $\mathfrak{m}$ in the set $\mathcal{S}$ of natural density $1$.
\end{theorem}
A natural question is whether the exponent $-1/2$ in the above can be improved further for a subset $\mathcal{S}$ of density $1$. In this direction, in \cite[Corollary $5$]{Luka-Radziwill-Shaperlinski}, Luca, Radziwi\l{}\l{} and Shparlinski proved the following result for Fourier coefficients of elliptic cusp forms. 
\begin{theorem}[Luca, Radziwi\l{}\l{} and Shparlinski]
Let $f$ be an elliptic cusp form which is a normalized Hecke eigenform of weight $k$ for the full modular group with Fourier coefficients $c_{f}(n)$. Assume that the symmetric power $L$-functions of $f$ are automorphic and satisfy the Langlands functoriality conjecture.
Then, for fixed $\upsilon\in\Bbb R$, we have
$$
\lim_{x\rightarrow\infty}\frac{\#\bigg\{n\leq x: c_{f}(n)\neq 0, \frac{\log |c_{f}(n)|+\frac{1}{2}\log\log n}{\sqrt{(\frac{1}{2}+\frac{\pi^{2}}{12})\log\log n}}\geq\upsilon\bigg\}}{\#\{n\leq x:c_{f}(n)\neq 0\}}=\frac{1}{\sqrt{2\pi}}\int_{\upsilon}^{\infty}e^{-u^{2}/2} du.
$$
\end{theorem}
The automorphy of symmetric powers of $f$ have been proved in a recent preprint of Newton and Thorne (see \cite{Newton-Thorne 1}) and hence by now the above theorem is unconditional.
It is clear from the above theorem that the exponent $-1/2$ cannot be improved further for a  density $1$ subset of $\Bbb N$. However, our next theorem is the following:
\begin{theorem}\label{infiniteDiophantine}
Let $\bf f$ be a primitive cusp form in $\mathscr{F}_{k}(\mathfrak{n})$ of weight $k=(k_{1},\ldots,k_{m})$ such that $k_{1}\equiv\cdots\equiv k_{m}\equiv 0\pmod{2}$ and each $k_{j}\geq 2$. Suppose $\bf f$ does not have complex multiplication. Then there exists a positive constant $\Lambda_{{\bf f}}$ depending only on $\bf f$ such that
$$
0<|c_{{\bf f}}(\mathfrak{m})|\leq\frac{\Lambda_{{\bf f}}}{\log{\N}(\mathfrak{m})},  
$$
holds for infinitely many integral ideals $\mathfrak{m}$. 
\end{theorem}
The above theorem is one of the main source of inspiration of this article. Before we state our next theorem, we first define the notion of relative density. Let $\mathscr{P}$ denote the set of all integral prime ideals of $F$. We say that a subset of prime ideals $\wp$ has the relative density $d(\wp)$ if 
$$
d(\wp)=\lim_{x\rightarrow\infty}\frac{\#\{{\N}(\mathfrak{p})\leq x:\mathfrak{p}\in\wp\}}{\#\{{\N}(\mathfrak{p})\leq x:\mathfrak{p}\in\mathscr{P}\}}.
$$
Then we have the following refinement of the above theorem.
\begin{theorem}\label{primepowerdiophantine}
	Let $\bf f$ be a primitive cusp form in $\mathscr{F}_{k}(\mathfrak{n})$ of weight $k=(k_{1},\ldots,k_{m})$ such that $k_{1}\equiv\cdots\equiv k_{m}\equiv 0\pmod{2}$ and each $k_{j}\geq 2$. Suppose $\bf f$ does not have complex multiplication. Then there exists a set $\wp$ of prime ideals of relative densiy 1 such that for each $\mathfrak{p}\in\wp$, there are infinitely many positive integers $n$ with
	\begin{equation}\label{primepowerbound}
	0<|c_{{\bf f}}(\mathfrak{p}^{n})|\leq\frac{\Lambda_{{\bf f}}}{n+1} 
	\end{equation}
	where $\Lambda_{{\bf f}}$ is a positive constant depending only on $\bf f$.
\end{theorem}
It is an interesting question to investigate whether the right-hand side of \eqref{primepowerbound} can be improved further. In this direction, we have the following theorem under the assumption of Littlewood's conjecture.
\begin{theorem}\label{Littlewoodtwoprime}
Let $\bf f$ be a primitive cusp form in $\mathscr{F}_{k}(\mathfrak{n})$ of weight $k=(k_{1},\ldots,k_{m})$ such that $k_{1}\equiv\cdots\equiv k_{m}\equiv 0\pmod{2}$ and each $k_{j}\geq 2$. Suppose $\bf f$ does not have complex multiplication and Conjecture \ref{Littlewood} is true. Then there exists a set $\wp$ of prime ideals of relative density 1 such that for any two distinct prime ideals $\mathfrak{p}$ and $\mathfrak{q}$ in $\wp$, we have $c_{{\bf f}}(\mathfrak{p}^{n}\mathfrak{q}^{n})\neq 0$, for each $n\geq 1$ and
$$
\liminf_{n\rightarrow\infty}\ (n+1)\cdot |c_{{\bf f}}(\mathfrak{p}^{n}\mathfrak{q}^{n})|=0.
$$
Furthermore, if instead of Conjecture \ref{Littlewood}, we assume Conjecture \ref{Velani-Badziahin}. Then there exists a set $\wp$ of prime ideals of relative density 1 such that for any two distinct prime ideals $\mathfrak{p}$ and $\mathfrak{q}$ in $\wp$, we have 
$$
\liminf_{n\rightarrow\infty}\ (n+1)\cdot\log (n+1)\cdot |c_{{\bf f}}(\mathfrak{p}^{n}\mathfrak{q}^{n})|<\infty.
$$
\end{theorem}

Finallly, we prove the following lower bound for $|c_{{\bf f}}(\mathfrak{p}^{n})|$ provided the corresponding Hecke eigen angle at $\mathfrak{p}$ is badly approximable.
\begin{theorem}\label{primepowerlowerthm}
	Let $\bf f$ be a primitive cusp form in $\mathscr{F}_{k}(\mathfrak{n})$ of weight $k=(k_{1},\ldots,k_{m})$ such that $k_{1}\equiv\cdots\equiv k_{m}\equiv 0\pmod{2}$ and each $k_{j}\geq 2$. If the Hecke eigen angle $\theta_{\mathfrak{p}}/\pi$ is badly approximable then for each positive integer $n$, we have 
	\begin{equation}
	|c_{{\bf f}}(\mathfrak{p}^{n})|>\frac{\eta_{{\bf f}}}{n+1}
	\end{equation}
	where $\eta_{\bf f}$ is a positive constant depending only on $\bf f$.
\end{theorem}
\begin{remark}
The above theorem have various consequences if the Hecke eigen angle $\theta_{\mathfrak{p}}/\pi$ is badly approximable, which we shall discuss in Remark \ref{rmkafterprf}.
\end{remark}
The article is organized as follows. In the next section, we recall the definition and certain properties of Hilbert modular forms which are needed to prove our theorems. We also prove estimates for the moments of the Fourier coefficients of Hilbert cusp forms and discuss the results on Diophantine approximation in this section. Then in the subsequent sections, we prove the theorems.
\section{Priliminaries}\label{priliminaries}
\subsection{Hilbert cusp forms: } In this subsection, we recall the definitions and notations for modular forms associated to a totally real number fields. For the theory of Hilbert modular forms, we refer to \cite{Shimura}.
Let $F$ be a totally real number field of degree $m$, $\mathcal{O}_{F}$ the ring of integers of $F$ and $\mathfrak{D}_{F}^{-1}$ the inverse of the different ideal $\mathfrak{D}_{F}$ of $F$. We denote the set of all fractional ideals of $F$ by $\mathcal{I}$. The norm of an ideal $\mathfrak{a}$ in $\mathcal{O}_{F}$ is defined as ${\N}(\mathfrak{a})=[\mathcal{O}_{F}:\mathfrak{a}]$, the index of $\mathfrak{a}$ in $\mathcal{O}_{F}$. The real embeddings of $F$ are denoted by $\sigma_{i}$ with $\sigma_{i}(x):=x_{i}$, $i=1,\ldots,m$ and $x\in F$. We say an element $x$ in $F$ is totally positive if $x_{i}>0$ for all $i$ and write $x\gg 0$.
We denote the narrow class number of $F$ by $h^{+}$ which is the cardinality of the narrow class group. The narrow class group is the quotient group $\mathcal{I}/\mathcal{P}^{+}$ where $\mathcal{P}^{+}$ is the group of all principal fractional ideals of $F$ generated by totally positive elements. Let ${\rm GL}_{2}(F)$ denote the group of invertible matrices of size $2$ with entries in $F$ and ${\rm GL}_{2}^{+}(F)$ be the subset of ${\rm GL}_{2}(F)$ with totally positive determinant. Let $\Bbb H$ denote the complex upper half plane. Then ${\rm GL}_{2}^{+}(F)$ acts on $\Bbb H^{n}$ by the fractional linear transformation
$$
\begin{pmatrix}
a & b\\
c & d
\end{pmatrix}\cdot (\tau_{1},\ldots,\tau_{m})=\left(\frac{a_{1}\tau_{1}+b_{1}}{c_{1}\tau_{1}+d_{1}},\cdots,\frac{a_{m}\tau_{m}+b_{m}}{c_{m}\tau_{m}+d_{m}}\right).
$$

Next we recall the definition and basic properties of Hilbert modular forms from \cite{Shimura}.
Let $\mathfrak{p}$ be a non-archimedean place of $F$ and $F_{\mathfrak{p}}$ be the completion of $F$ with respect to the valuation defined by $\mathfrak{p}$. For an integral ideal $\mathfrak{c}$ of $F$, one defines a subgroup $K_{\mathfrak{p}}(\mathfrak{c})$ of ${\rm GL}_{2}(F_{\mathfrak{p}})$ as 
$$
K_{\mathfrak{p}}(\mathfrak{c})=
\left\{\begin{pmatrix}
a & b\\
c & d
\end{pmatrix}\in {\rm GL}_{2}(F_{\mathfrak{p}}): 
\begin{matrix}
a\in\mathcal{O}_{F_{\mathfrak{p}}}, & b\in\mathfrak{D}_{F_{\mathfrak{p}}}^{-1},\\ 
c\in\mathfrak{c}\mathfrak{D}_{F_{\mathfrak{p}}}, & d\in\mathcal{O}_{F_{\mathfrak{p}}}, 
\end{matrix}\ |ad-bc|_{\mathfrak{p}}=1\right\},
$$
where the subscript $\mathfrak{p}$ means the $\mathfrak{p}$-part of given ideals in $F_{\mathfrak{p}}$. Moreover, we put
$$
K_{0}(\mathfrak{c})={\rm SO}(2)^{n}\cdot\prod_{\mathfrak{p}<\infty}K_{\mathfrak{p}}(\mathfrak{c})\ {\rm and\ }W(\mathfrak{c})={\rm GL}_{2}^{+}(\Bbb R)^{n}K_{0}(\mathfrak{c}).
$$
Then we have the following decomposition of ${\rm GL}_{2}(\Bbb A_{F})$
\begin{equation}\label{decompose}
{\rm GL}_{2}(\Bbb A_{F})=\cup_{\nu=1}^{h^{+}}{\rm GL}_{2}({F})x_{\nu}^{-\iota}W(\mathfrak{c}),
\end{equation}
where $x_{\nu}^{-\iota}=\begin{pmatrix}
g_{\nu}^{-1} &0\\
0 & 1
\end{pmatrix}$
with $\{g_{\nu}\}_{\nu=1}^{h^+}$ taken to be a complete set of the representatives of the narrow class group of $F$. 
Let $g_{\nu}$ be a representative of the narrow class group. Then one defines a congruence subgroup of ${\rm GL}_{2}(F)$ as 
\begin{eqnarray*}
\Gamma_{\nu}(\mathfrak{c})&=&{\rm GL}_{2}(F)\cap x_{\nu} W(\mathfrak{c})x_{\nu}^{-1}\\
&=&\left\{\begin{pmatrix}
a & g_{\nu}^{-1}b\\
g_{\nu}c & d
\end{pmatrix}\in {\rm GL}_{2}^{+}(F): 
\begin{matrix}
a\in\mathcal{O}_{F}, & b\in\mathfrak{D}_{F}^{-1},\\ 
c\in\mathfrak{c}\mathfrak{D}_{F}, & d\in\mathcal{O}_{F}, 
\end{matrix}\ ad-bc\in\mathcal{O}_{F}^{\times}\right\}.
\end{eqnarray*}
Let $\mathfrak{n}$ be an integral ideal of $F$ and $k:=(k_{1},\ldots,k_{m})\in\Bbb Z_{+}^{m}$. Then we define $M_{k}(\Gamma_{\nu}(\mathfrak{n}))$ to be the space of all complex valued holomorphic functions $f$ on $\Bbb H^{m}$ and at the cusps such that
$$
f\Vert_{k} A(\tau)=f
$$
where
$$
f\Vert_{k} A(\tau)=(\det A)^{k/2}(c\tau+d)^{-k}f(A\cdot\tau),
$$
 $A=\begin{pmatrix}
a & b\\
c & d
\end{pmatrix}\in\Gamma_{\nu}(\mathfrak{n})$, $\tau=(\tau_{1},\ldots,\tau_{m})\in\Bbb H^{m}$ with $(\det A)^{k/2}=\prod_{j=1}^{m}(a_{j}d_{j}-b_{j}c_{j})^{k_{j}/2}$ and $(c\tau+d)^{k}=\prod_{j=1}^{m}(c_{j}\tau_{j}+d_{j})^{k_{j}}$.

A Hilbert modular form $f_{\nu}$ in $M_{k}(\Gamma_{\nu}(\mathfrak{n}))$ has a Fourier expansion of the form
$$
f_{\nu}(\tau)=\sum_{\substack{0\ll\xi\in g_{\nu}\mathcal{O}_{F}\\ \xi=0}}a_{\nu}(\xi)\exp\left(2\pi i\sum_{j=1}^{m}\xi_{j}\tau_{j}\right).
$$
A Hilbert modular form $f$ in $M_{k}(\Gamma_{\nu}(\mathfrak{n}))$ is called a cusp form if for all $\alpha\in{\rm GL}_{2}^{+}(F)$, the constant term in the Fourier expansion of $f\Vert_{k}\alpha$ is $0$.

We write $\bf f$ for a collection $(f_{1},\ldots,f_{h^{+}})$ of Hilbert modular forms $f_{\nu}$ in $M_{k}(\Gamma_{\nu}(\mathfrak{n}))$, $(\nu=1,\dots,h^{+})$ and we identify it with a function on ${\rm GL}_{2}(\Bbb A_{F})$ by
$$
{\bf f}(g)={\bf f}(\alpha x_{\nu}^{-\iota}w)=(f_{\nu}\Vert_{k} w_{\infty})(\bf i)
$$
where $\alpha x_{\nu}^{-\iota}w\in {\rm GL}_{2}({F})x_{\nu}^{-\iota}W(\mathfrak{n})$ as in \eqref{decompose} and ${\bf i}=(i,\ldots,i)$ (with $i=\sqrt{-1}$). The space of such functions $\bf f$ is denoted as 
$$
\mathfrak{M}_{k}(\mathfrak{n})=\prod_{\nu=1}^{h^{+}}M_{k}(\Gamma_{\nu}(\mathfrak{n})).
$$
If each $f_{\nu}$ is a cusp form then the space of such $\bf f$ in $\mathfrak{M}_{k}(\mathfrak{n})$ is denoted by $\mathfrak{S}_{k}(\mathfrak{n})$.
We write $k_{0}:=\max\{k_1,\ldots,k_{m}\}$. One can associate the Fourier coefficients $\{a_{\nu}(\xi)\}$ for each $f_{\nu}$ with the function $\bf f$. Let $\mathfrak{m}$ be an integral ideal of $F$. Then we have $\mathfrak{m}=\xi g_{\nu}^{-1}\mathcal{O}_{F}$ for a unique $\nu$ and some totally positive element $\xi$ in $F$.  For each integral ideal $\mathfrak{m}$ of $F$, we define
$$
C_{{\bf f}}(\mathfrak{m})=\begin{cases}
a_{\nu}(\xi)\xi^{-k/2} {\N}(\mathfrak{m})^{k_{0}/2} & {\rm if }\ \mathfrak{m}=\xi g_{\nu}^{-1}\mathcal{O}_{F}\ {\rm and }\ \mathfrak{m}\ {\rm is\ integral},\\
0 & {\rm if }\ \mathfrak{m}\ {\rm is\ not\ integral}.
\end{cases}
$$
The above expression is independent of the choice of $\xi$ up to totally positive elements in $\mathcal{O}_{F}^{\times}$ and thus makes sense.

We also refer to \cite[Chapter 1,2]{Garrett} and \cite{Ghate, Ghate1} for more details on Hilbert modular forms. The space $\mathfrak{S}_{k}(\mathfrak{n})$ can be decomposed as $\mathfrak{S}_{k}(\mathfrak{n})=\mathfrak{S}_{k}^{old}(\mathfrak{n})\oplus \mathfrak{S}_{k}^{new}(\mathfrak{n})$ where $\mathfrak{S}_{k}^{old}(\mathfrak{n})$ is the subspace of old forms that come from lower levels. The space $\mathfrak{S}_{k}^{new}(\mathfrak{n})$ is the orthogonal complement of $\mathfrak{S}_{k}^{old}(\mathfrak{n})$ with respect to the Petersson scalar product. The space $\mathfrak{S}_{k}(\mathfrak{n})$ is invariant under the action of Hecke operators $\{T_{\mathfrak{m}}\}_{\mathfrak{m}\subset\mathcal{O}_{F}}$.

A Hilbert cusp form $\bf f$ in $\mathfrak{S}_{k}^{new}(\mathfrak{n})$ is said to be primitive if it is a normalized (i.e. $C_{\bf f}(\mathcal{O}_{F})=1$) common Hecke eigenfunctions of all Hecke operators. we refer to \cite[p. 381-82]{Shemanske-Walling}. We denote by $\mathscr{F}_{k}(\mathfrak{n})$ the set of all primitive forms of weight $k$ and level $\mathfrak{n}$. If $\bf f$ is a form in $\mathscr{F}_{k}(\mathfrak{n})$ then it follows from the work of \cite{Shimura} that the coefficients $C_{{\bf f}}(\mathfrak{m})$ are the same as the Hecke eigenvalues of $T_{\mathfrak{m}}$ for all $\mathfrak{m}$. Moreover, the coefficients $C_{{\bf f}}(\mathfrak{m})$ are known to be real numbers for all $\mathfrak{m}$, since the form $\bf f$ is with the trivial character. 

Throughout the paper, we denote by $c_{\bf f}(\mathfrak{m}):=\frac{C_{\bf f}(\mathfrak{m})}{{\N}(\mathfrak{m})^{\frac{k_{0}-1}{2}}}$ the normalized coefficient of $C_{\bf f}(\mathfrak{m})$. Let $\bf f$ be a primitive cusp form in $\mathscr{F}_{k}(\mathfrak{n})$ of weight $k=(k_{1},\ldots,k_{m})$ such that $k_{1}\equiv\cdots\equiv k_{m}\equiv 0\pmod{2}$ and each $k_{j}\geq 2$. Suppose that $\bf f$ does not have complex multiplication.
It is known from \cite{Shimura} that the coefficients $c_{\bf f}(\mathfrak{m})$ satisfy
\begin{eqnarray}\label{multiplicative}
c_{\bf f}(\mathfrak{mn})&=&c_{\bf f}(\mathfrak{m})c_{\bf f}(\mathfrak{n})\ \ \text{\rm if }(\mathfrak{m},\mathfrak{n})=1,\nonumber\\
c_{\bf f}(\mathfrak{p}^{m})&=&c_{\bf f}(\mathfrak{p})c_{\bf f}(\mathfrak{p}^{m-1})-c_{\bf f}(\mathfrak{p}^{m-2})\ \ \text{\rm for }\mathfrak{p}\ \text{\rm a prime ideal},\label{Heckerecursion}\\
|c_{\bf f}(\mathfrak{m})|&\leq& d(\mathfrak{m}),\nonumber
\end{eqnarray}
where $d(\mathfrak{m})$ is the divisor function. The above inequality was proven in \cite{Blasius}. For any prime ideal $\mathfrak{p}$, we define the angle $\theta_{\mathfrak{p}}\in[0,\pi]$ such that 
$$
c_{\bf f}(\mathfrak{p})=2\cos\theta_{\mathfrak{p}}.
$$
It follows from \cite{Blasius} that for a prime ideal $\mathfrak{p}\nmid\mathfrak{n}\mathfrak{D}_{F}$,
$$
c_{\bf f}(\mathfrak{p})=\mu_{\mathfrak{p}}+\overline{\mu_{\mathfrak{p}}}
$$
for some $\mu_{\mathfrak{p}}\in\Bbb C$ with $|\mu_{\mathfrak{p}}|=1$. Using the Hecke relation \eqref{Heckerecursion}, it follows by the induction on $m$ that
\begin{equation}\label{cfpmexpression}
c_{\bf f}(\mathfrak{p}^{m})=\frac{\mu_{\mathfrak{p}}^{m+1}-\overline{\mu_{\mathfrak{p}}}^{m+1}}{\mu_{\mathfrak{p}}-\overline{\mu_{\mathfrak{p}}}}=\frac{\sin (m+1)\theta_{\mathfrak{p}}}{\sin \theta_{\mathfrak{p}}}.
\end{equation}
The above relation will play an important role in proof of our theorems. 
\subsection{Sato-Tate equidistribution} In this subsection, we recall the Sato-Tate equidistribution theorem for Hilbert modular forms of integral weight without CM which will play a crucial role in proof of our theorems. For this, we refer to \cite[Corollary 7.1.7]{Lamb-Gee-Geraghty}.  Here we state the theorem in terms of the Fourier coefficients of primitive Hilbert cusp forms from \cite[Theorem 3.3]{Kaushik-Kumar-Tanabe}.
\begin{theorem}[Sato-Tate for Hilbert cusp forms]\label{Satitatehilbert}
Let $\bf f$ be a primitive cusp form in $\mathscr{F}_{k}(\mathfrak{n})$ of weight $k=(k_{1},\ldots,k_{m})$ such that $k_{1}\equiv\cdots\equiv k_{m}\equiv 0\pmod{2}$ and each $k_{j}\geq 2$. Suppose that $\bf f$ does not have complex multiplication. Then, for any prime ideal $\mathfrak{p}$ of $F$ such that $\mathfrak{p}\nmid{\mathfrak{n}\mathfrak{D}_{F}}$, we have $c_{{\bf f}}(\mathfrak{p})\in [-2,2]$.
Furthermore, $\{c_{\bf f}(\mathfrak{p})\}_{\mathfrak{p}}$ are equidistributed in $[-2,2]$ with respect to the measure $\mu=(2/\pi)\sqrt{1-t^2}dt$. In other words, for any subinterval $I$ of $[-2,2]$, we have
$$
\lim_{x\rightarrow\infty}\frac{\#\{\mathfrak{p}\in\Bbb P:\mathfrak{p}\nmid{\mathfrak{n}\mathfrak{D}_{F}},{\N}(\mathfrak{p})\leq x, c_{\bf f}(\mathfrak{p})\in I\}}{\#\{\mathfrak{p}\in\Bbb P:{\N}(\mathfrak{p})\leq x\}}=\mu(I)=\frac{2}{\pi}\int_{I}\sqrt{1-t^2}dt.
$$
\end{theorem}

We also need the information about the distribution of the angles $\theta_{\mathfrak{p}}/2\pi$ that are irrational. In fact, we have the following result.
\begin{lemma}\label{thetapnotrational}
The angles $\theta_{\mathfrak{p}}/2\pi$ that are irrational are equidistributed in the interval $[0,1/2]$.
\end{lemma}
\begin{proof}
Proof of the above lemma is based on the known procedure (see \cite{Bengoechea, Murty-Murty}). However we give proof for convenience. 
It is clear from Theorem \ref{Satitatehilbert} that the angles $\theta_{\mathfrak{p}}$ are equidistributed in the interval $[0,\pi]$. Therefore it suffices to show that the angles $\theta_{\mathfrak{p}}/2\pi$ that are rational in fact forms a finite set. Let $\mathfrak{p}$ be a prime ideal such that $\theta_{\mathfrak{p}}/2\pi$ is a rational number. We write $\theta_{\mathfrak{p}}/2\pi= A_{\mathfrak{p}}/B_{\mathfrak{p}}$ where $B_{\mathfrak{p}}\geq 1$, $A_{\mathfrak{p}}<B_{\mathfrak{p}}$ and $(A_{\mathfrak{p}},B_{\mathfrak{p}})=1$. Since $\mu_{\mathfrak{p}}=e^{2\pi i(A_{\mathfrak{p}}/B_{\mathfrak{p}})}$ is a primitive $B_{\mathfrak{p}}$-th root of unity, we have $[\Bbb Q(\mu_{\mathfrak{p}}):\Bbb Q]=\phi(B_{\mathfrak{p}})$, where $\phi$ is the Euler function. On the other hand, $\mu_{\mathfrak{p}}$ is a zero of the polynomial 
\begin{eqnarray*}
P(z)&=&(z-\mu_{\mathfrak{p}})(z-\overline{\mu_{\mathfrak{p}}})(z+\mu_{\mathfrak{p}})(z+\overline{\mu_{\mathfrak{p}}})\\
&=&(z^{2}-c_{\bf f}(\mathfrak{p})z+1)(z^{2}+c_{\bf f}(\mathfrak{p})z+1)=z^{4}+(2-c_{\bf f}(\mathfrak{p})^{2})z^{2}+1.
\end{eqnarray*}
It is known from \cite[Proposition 2.8, p.654]{Shimura} that the field $\Bbb K$ generated by the coefficients of the primitive forms in $\mathscr{F}_{k}(\mathfrak{n})$ is a finite extension over $\Bbb Q$ and let $[\Bbb K:\Bbb Q]=d$. Since the coefficients of the polynomial $P(z)$ is in $\Bbb K$, therefore $[\Bbb Q(\mu_{\mathfrak{p}}):\Bbb Q]\leq 4d$. It follows that $\phi(B_{\mathfrak{p}})\leq 4d$ and thus $B_{\mathfrak{p}}$ belongs to a finite set of integers independent of $\mathfrak{p}$. Since $A_{\mathfrak{p}}<B_{\mathfrak{p}}$, therefore $\theta_{\mathfrak{p}}/2\pi$ belongs to a finite set of rationals. Hence the lemma follows.
\end{proof}

\subsection{Estimates for the moments of the Fourier coefficients of Hilbert cusp forms}
In this subsection, we prove the estimates for the moments of the Fourier coefficients of Hilbert cusp forms which play an important role in proof of Theorem \ref{Hilhalfepsilon}. We first prove the following lemma.
\begin{lemma}\label{generalsumesti}
Let $f:\mathcal{I}\rightarrow\Bbb R$ be a non-negative multiplicative function. Assume that there exist positive constants $A$ and $B$ such that for $x>1$ both inequalities
\begin{eqnarray}
&&\sum_{{\N}(\mathfrak{p})\leq x}f(\mathfrak{p})\log{\N}(\mathfrak{p})\leq A x\ \  \text{ and }\label{sumfplogp}\\ &&\sum_{\mathfrak{p}}\sum_{\alpha\geq 2}\frac{f(\mathfrak{p}^{\alpha})}{{\N}(\mathfrak{p}^{\alpha})}\log({\N}(\mathfrak{p}^{\alpha}))\leq B\label{sumfplogphigher}
\end{eqnarray}
hold. Then we have
\begin{equation}\label{sumfn}
\sum_{{\N}(\mathfrak{n})\leq x}f(\mathfrak{n})\leq (A+B+1)\frac{x}{\log x}\sum_{{\N}(\mathfrak{n})\leq x}\frac{f(\mathfrak{n})}{{\N}(\mathfrak{n})}.
\end{equation}
\end{lemma}
\begin{proof}
	Proof of the lemma follows exactly the similar arguments of \cite[Theorem 01]{Hall-Tanenbaum} or \cite[Lemma 9.6]{Koninck-Luca}. However, we prove it for completeness.
	
	Let $\mathcal{S}(x)$ and $\mathcal{L}(x)$ be the sum which appears on the left-hand side and the right-hand side of \eqref{sumfn} respectively. It is easy to see that $\mathcal{S}(x)\leq x\mathcal{L}(x)$.
	It clearly follows that 
	\begin{eqnarray}\label{sxlogx}
	\mathcal{S}(x)\log x&=&	\sum_{{\N}(\mathfrak{n})\leq x}f(\mathfrak{n})\log x=\sum_{{\N}(\mathfrak{n})\leq x}f(\mathfrak{n})\log \left(\frac{x}{{\N}(\mathfrak{n})}\right)+\sum_{{\N}(\mathfrak{n})\leq x}f(\mathfrak{n})\sum_{\mathfrak{p}\parallel\mathfrak{n}}\log {\N}(\mathfrak{p})\nonumber\\
	&+&\sum_{{\N}(\mathfrak{n})\leq x}f(\mathfrak{n})\sum_{\substack{\alpha\geq 2\\ \mathfrak{p}^{\alpha}\parallel\mathfrak{n}}}\log({\N}(\mathfrak{p}^{\alpha})):=\mathcal{S}_{1}+\mathcal{S}_{2}+\mathcal{S}_{3},\ \ (\text{ say}).
	\end{eqnarray}
	We estimate the sums $\mathcal{S}_{1}$, $\mathcal{S}_{2}$ and $\mathcal{S}_{3}$ separately.
	From the inequality $\log \left(\frac{x}{{\N}(\mathfrak{n})}\right)<\frac{x}{{\N}(\mathfrak{n})}$ for all integral ideals $\mathfrak{n}$ with ${\N}(\mathfrak{n})\leq x$, we see that
	\begin{equation}\label{s1}
	\mathcal{S}_{1}\leq \sum_{{\N}(\mathfrak{n})\leq x}f(\mathfrak{n}) \frac{x}{{\N}(\mathfrak{n})}=x\mathcal{L}(x).
	\end{equation}
	Next we estimate the sum $\mathcal{S}_{2}$. Let $\mathfrak{n}$ be an integral ideal that appears in the sum $\mathcal{S}_{2}$. Then writing $\mathfrak{n}=\mathfrak{m}\mathfrak{p}$ with  $(\mathfrak{p},\mathfrak{m})=1$ in $\mathcal{S}_{2}$, changing the order of summation and using inequality \eqref{sumfplogp}, we get
	\begin{equation}\label{s2}
	\mathcal{S}_{2}=\sum_{{\N}(\mathfrak{m})\leq x}f(\mathfrak{m})\sum_{\substack{{\N}(\mathfrak{p})\leq x/{{\N}(\mathfrak{m})}\\ \mathfrak{p}\nmid\mathfrak{m}}}f(\mathfrak{p})\log {\N}(\mathfrak{p})\leq A\sum_{{\N}(\mathfrak{m})\leq x}f(\mathfrak{m})\frac{x}{{\N}(\mathfrak{m})}=Ax\mathcal{L}(x).
	\end{equation}
	In order to estimate $\mathcal{S}_{3}$, for each $\alpha\geq 2$ and $\mathfrak{p}^{\alpha}\!\!\parallel\!\!\mathfrak{n}$, we write $\mathfrak{n}=\mathfrak{m}\mathfrak{p}^{\alpha}$ with $(\mathfrak{p},\mathfrak{m})=1$ in $\mathcal{S}_{3}$,changing the order of the summation and using inequality \eqref{sumfplogphigher}, we get
	\begin{eqnarray}
	\mathcal{S}_{3}&=&\sum_{\mathfrak{p}}\sum_{\alpha\geq 2}{f(\mathfrak{p}^{\alpha})}\log({\N}(\mathfrak{p}^{\alpha}))\sum_{\substack{{\N}(\mathfrak{m})\leq x/{{\N}(\mathfrak{p}^{\alpha})}\\ \mathfrak{p}\nmid\mathfrak{m}}}f(\mathfrak{m})\nonumber\\
	&\leq&\sum_{\mathfrak{p}}\sum_{\alpha\geq 2}{f(\mathfrak{p}^{\alpha})}\log({\N}(\mathfrak{p}^{\alpha}))\mathcal{S}\left(\frac{x}{{\N}(\mathfrak{p}^{\alpha})}\right)\nonumber\\
	&\leq&\sum_{\mathfrak{p}}\sum_{\alpha\geq 2}{f(\mathfrak{p}^{\alpha})}\log({\N}(\mathfrak{p}^{\alpha}))\cdot\frac{x}{{\N}(\mathfrak{p}^{\alpha})}\cdot\mathcal{L}\left(\frac{x}{{\N}(\mathfrak{p}^{\alpha})}\right)\nonumber\\
	&\leq&x\mathcal{L}(x)\sum_{\mathfrak{p}}\sum_{\alpha\geq 2}\frac{f(\mathfrak{p}^{\alpha})}{{\N}(\mathfrak{p}^{\alpha})}\log({\N}(\mathfrak{p}^{\alpha}))\nonumber\\
	&\leq&B x \mathcal{L}(x).\label{s3}
	\end{eqnarray}
	Inequality \eqref{sumfn} now follows from \eqref{s1}, \eqref{s2} and \eqref{s3} and hence proof is done.
\end{proof}
Now, we prove the lemma which will be used to prove Theorem \ref{Hilhalfepsilon}.
\begin{lemma}\label{momentestimate}
	Let $\bf f$ be a primitive cusp form in $\mathscr{F}_{k}(\mathfrak{n})$ of weight $k=(k_{1},\ldots,k_{m})$ such that $k_{1}\equiv\cdots\equiv k_{m}\equiv 0\pmod{2}$ and each $k_{j}\geq 2$. Suppose that $\bf f$ does not have complex multiplication and	
	let $c_{{\bf f}}(\mathfrak{m})$ be the normalized  coefficients of $\bf f$. Then as $x\rightarrow\infty$, we have
	\begin{eqnarray}
	&&\sum_{{\N}(\mathfrak{n})\leq x}\frac{|c_{\bf f}(\mathfrak{n})|}{{\N}(\mathfrak{n})}\ll (\log x)^{0.85};\label{sumcfn1}\\
	&&\sum_{{\N}(\mathfrak{n})\leq x}|c_{\bf f}(\mathfrak{n})|^{2}\leq x(\log x)^{o(1)};\label{sumcfn2}\\
	&&\sum_{{\N}(\mathfrak{n})\leq x}|c_{\bf f}(\mathfrak{n})|^{\gamma}\leq x(\log x)^{-\gamma/2+c\gamma^{2}}\label{sumcfngama}
	\end{eqnarray}
	where $\gamma<1$ and $c>0$ is an absolute constant.
\end{lemma}
\begin{proof}
	It follows from Theorem \ref{Satitatehilbert} that $c_{\bf f}(\mathfrak{p})=2\cos\theta_{\mathfrak{p}}$ with the angles $\theta_{\mathfrak{p}}\in[0,\pi]$. Also, the coefficients $c_{\bf f}(\mathfrak{p^{a}})$ for $a\geq 2$ satisfy the Ramanujan-Petersson bound (proved by Blasius \cite{Blasius}), 
	\begin{equation}\label{ramanujanprimepower}
	c_{\bf f}(\mathfrak{p}^{a})\leq (a+1)\leq {\N}(\mathfrak{p})^{(a-1)/2}(\log {\N}(\mathfrak{p}))^{-1-\rho}
	\end{equation}
	where $\rho>0$ is a constant. Using the multiplicativity of $c_{\bf f}(\mathfrak{n})$ and by the bound \eqref{ramanujanprimepower}, we get for any $\gamma\in[0,2]$,
	\begin{eqnarray}
	\sum_{{\N}(\mathfrak{n})\leq x}\frac{|c_{\bf f}(\mathfrak{n})|^{\gamma}}{{\N}(\mathfrak{n})}&\leq&\prod_{{\N}(\mathfrak{p})\leq x}\left(1+\frac{|c_{\bf f}(\mathfrak{p})|^{\gamma}}{{\N}(\mathfrak{p})}+\frac{|c_{\bf f}(\mathfrak{p}^{2})|^{\gamma}}{{\N}(\mathfrak{p}^{2})}+\cdots\right)\nonumber\\
	&=&\prod_{{\N}(\mathfrak{p})\leq x}\left(1+\frac{|2\cos\theta_{\mathfrak{p}}|^{\gamma}}{{\N}(\mathfrak{p})}+O\left(\frac{1}{{\N}(\mathfrak{p})(\log {\N}(\mathfrak{p}))^{1+\rho}}\right)\right)\nonumber\\
	&=&\exp\left(\sum_{{\N}(\mathfrak{p})\leq x}\frac{|2\cos\theta_{\mathfrak{p}}|^{\gamma}}{{\N}(\mathfrak{p})}+O\left(\sum_{{\N}(\mathfrak{p})\geq 2}\frac{1}{{\N}(\mathfrak{p})(\log {\N}(\mathfrak{p}))^{1+\rho}}\right)\right)\nonumber\\
	&=&\exp\left(\sum_{{\N}(\mathfrak{p})\leq x}\frac{|2\cos\theta_{\mathfrak{p}}|^{\gamma}}{{\N}(\mathfrak{p})}+O(1)\right).\label{sumgamaestimate}
	\end{eqnarray}
	By the partial summation and Mertens formula \cite[Proposition 2, p. 339]{Lebacque}, we have by the Sato-Tate distribution (see Theorem \ref{Satitatehilbert}),
	\begin{equation}
	\sum_{{\N}(\mathfrak{p})\leq x}\frac{|2\cos\theta_{\mathfrak{p}}|^{\gamma}}{{\N}(\mathfrak{p})}=\left(\frac{2}{\pi}\int_{0}^{\pi}|2\cos\theta|^{\gamma}\sin^{2}\theta d\theta+o(1)\right)\log\log x
	\end{equation}
	as $x\rightarrow\infty$. Let 
	$$
	I(\gamma)=\frac{2}{\pi}\int_{0}^{\pi}|2\cos\theta|^{\gamma}\sin^{2}\theta d\theta.
	$$
	If $\gamma=1$ then the above integral becomes $0.848826\ldots$. Thus from \eqref{sumgamaestimate}, we see that the estimate \eqref{sumcfn1} follows. If $\gamma=2$ then the above integral equals 1. Furthermore, since $I^{'}(0)=-1/2$, we see that for $\gamma<1$
	$$
	I(\gamma)\leq 1-\frac{\gamma}{2}+O(\gamma^{2}).
	$$
	It follows from the above estimate that
	\begin{equation}\label{cfnratioestimate}
	\sum_{{\N}(\mathfrak{n})\leq x}\frac{|c_{\bf f}(\mathfrak{n})|^{2}}{{\N}(\mathfrak{n})}\leq(\log x)^{1+o(1)}\ \  \text{ \rm and }\ \sum_{{\N}(\mathfrak{n})\leq x}\frac{|c_{\bf f}(\mathfrak{n})|^{\gamma}}{{\N}(\mathfrak{n})}\leq(\log x)^{1-\frac{\gamma}{2}+c\gamma^{2}}
	\end{equation}
	where $c>0$ is an absolute constant.
	
	It can be seen that the functions $f(\mathfrak{n})=|c_{\bf f}(\mathfrak{n})|^{2}$ and $f(\mathfrak{n})=|c_{\bf f}(\mathfrak{n})|^{\gamma}$ satisfy the conditions of the Lemma \ref{generalsumesti} with some suitable constants $A$ and $B$. Hence, from \eqref{cfnratioestimate} and Lemma \ref{generalsumesti}, we deduce \eqref{sumcfn2} and \eqref{sumcfngama}. This completes  proof of the lemma.
\end{proof}

\subsection{Diophantine approximation} In this subsection, we discuss the results from the theory of homogeneous Diophantine approximation.
We first recall the fundamental result of Dirichlet in the Diophantine approximation which play a crucial role in proof of our theorems. 
\begin{lemma}[Dirichlet]\label{Dirichlet}
For any irrational number $\theta$, we have
$$
\parallel m\theta\parallel<1/m,
$$
for infinitely many positive integers $m$. Here $\parallel\! x\!\parallel$ denotes the distance of $x$ to the nearest integer. 
\end{lemma}

We also need the following conjectures in order to deduce Theorem \ref{Littlewoodtwoprime}.
\medskip

{\bf Littlewood's conjecture: } Let $\alpha,\beta$ be any two real numbers in $[0,1]$. Then by Dirichlet's theorem  and the fact that $\parallel\! x\!\parallel<1$ for any $x$, we have $\parallel\! n\alpha\!\parallel\!\cdot\!\parallel\! n\beta\!\parallel\leq n^{-1}$. A natural question is whether the statement remains valid if one replaces the right hand side of the inequality by $\epsilon n^{-1}$, for any arbitrary $\epsilon>0$. In this direction, Littlewood conjectured the following in the 1930's:

\begin{conjecture}[Littlewood's conjecture]\label{Littlewood}
Let $\alpha,\beta$ be any two real numbers in $[0,1]$. Then we have
$$
\liminf_{n\rightarrow\infty}n\!\parallel\! n\alpha\!\parallel\!\cdot\!\parallel\! n\beta\!\parallel=0.
$$
\end{conjecture}
The above conjecture lies very deep in the theory of homogeneous Diophantine approximation and still remains an outstanding open problem in its full generality.

Recently, in \cite[p. 2768]{Velani-Badziahin}, Velani and Badziahin proposed the following strengthening of Littlewood's conjecture.
\begin{conjecture}[Velani-Badziahin]\label{Velani-Badziahin}
Let $\alpha,\beta$ be any two real numbers. Then there exist infinitely many positive integers $n$ such that
$$
n\cdot\log n\cdot\parallel\! n\alpha\!\parallel\!\cdot\!\parallel\! n\beta\!\parallel\ll 1.
$$
\end{conjecture}
\subsection{Height function}\label{heightfunction} In this subsection, we define the height function on number fields which we need to discuss certain results in Remark \ref{rmkafterprf}. Let $K$ be any number field with $M_{K}$ the set of all inequivalent absolute values where each place $v$ is normalized so that the product formula holds. More precisely, for a finite place $\mathfrak{p}$ of $K$, we define the absolute value of $x\in K^{\times}$
$$
|x|_{\mathfrak{p}}={\N}(\mathfrak{p})^{-{\rm ord}_{\mathfrak{p}}(x)}
$$ 
where $x\mathcal{O}_{K}=\prod_{\mathfrak{p}}\mathfrak{p}^{{\rm ord}_{\mathfrak{p}}(x)}$ is the decomposition of the ideal $x\mathcal{O}_{K}$ into product of maximal ideals. We define ${\rm ord}_{\mathfrak{p}}(0)=\infty$.
For an infinite place $\mathfrak{p}_{\infty}$, there exists an embedding $\sigma:K\hookrightarrow\Bbb C$ such that for any $x\in K$
$$
|x|_{\mathfrak{p}_{\infty}}=|\sigma(x)|^{\frac{n_{\sigma}}{[K:\Bbb Q]}}
$$
where $n_{\sigma}$ is the local degree which equals $1$ or $2$ depending on whether $\sigma$ is real or complex embedding. If $x$ is contained in a number field $K$, then one defines its Weil height as 
$$
H(x)=\prod_{v\in M_{K}}\max\{1, |x|_{v}\}.
$$

\section{Proof of Theorem~\ref{Hilhalfepsilon}}
Let $\epsilon>0$ be fixed. We denote by 
$$
\mathcal{N}_{\epsilon}(x):=\#\{\mathfrak{n}\in\mathcal{I}:{\N}(\mathfrak{n})\leq x\text{ \rm and }|c_{\bf f}(\mathfrak{n})|>(\log {\N}(\mathfrak{n}))^{-1/2+\epsilon}\}.
$$
It follows from \eqref{sumcfngama} of Lemma \ref{momentestimate} that
$$
\mathcal{N}_{\epsilon}(x)\leq \sum_{{\N}(\mathfrak{n})\leq x}|c_{\bf f}(\mathfrak{n})|^{\gamma}\cdot (\log {\N}(\mathfrak{n}))^{\gamma/2-\gamma\epsilon}\ll x(\log x)^{-\gamma\epsilon+c\gamma^{2}},
$$
where $c>1$ is an absolute constant. Choosing $\gamma=\epsilon/(2c)$, we get
$$
\mathcal{N}_{\epsilon}(x)\leq x(\log x)^{-\epsilon^{2}/2}=o(x) \ \ \ \ \text{\rm as }x\rightarrow\infty.
$$
Hence, for any $\epsilon>0$, we see that the natural density of the set $\mathcal{N}_{\epsilon}(x)$ is 0. This completes proof of the theorem. 
\section{Proof of Theorem \ref{infiniteDiophantine}}

The proof rests on the ideas stem from \cite{Alkan-Ford-Zaharescu}.
For a prime ideal $\mathfrak{p}$ such that $\frac{\theta_{\mathfrak{p}}}{2\pi}$ is irrational,  it follows from Lemma \ref{Dirichlet} that the set $\{r\theta_{\mathfrak{p}}\}_{r\in\Bbb N}$ is a dense subset of the set of real numbers. Therefore the numbers 
$$
c_{\bf f}(\mathfrak{p}^{r})=\frac{\sin (r+1)\theta_{\mathfrak{p}}}{\sin \theta_{\mathfrak{p}}}.
$$
are dense in the interval $[-\frac{1}{\sin \theta_{\mathfrak{p}}},\frac{1}{\sin \theta_{\mathfrak{p}}}]$.
Therefore there exists a positive integer $r_{1}$ such that
$$
c_{\bf f}(\mathfrak{p}^{r_{1}}) > 0.
$$
Let $\mathfrak{q}$ be a prime ideal distinct from $\mathfrak{p}$ such that $\frac{\theta_{\mathfrak{q}}}{2\pi}$ is irrational. For any $n\geq1$, we work with the ideal $\mathfrak{a}_{n}=\mathfrak{p}^{r_{1}}\mathfrak{q}^{n-1}$ and we have $c_{\bf f}(\mathfrak{a}_{n})=c_{\bf f}(\mathfrak{p}^{r_{1}})c_{\bf f}(\mathfrak{q}^{n-{1}})$. Let 
$$
\delta=\frac{c_{\bf f}(\mathfrak{p}^{r_{1}})}{\sin \theta_{\mathfrak{q}}}.
$$
Then for any $n\geq 1$, we have $c_{\bf f}(\mathfrak{a}_{n})=\delta\sin n\theta_{\mathfrak{q}}$. 
By Lemma \ref{Dirichlet}, there exists infinitely many $n$ for which 
$$
\bigg\Vert n\frac{\theta_{\mathfrak{q}}}{2\pi}\bigg\Vert=\bigg\vert n\frac{\theta_{\mathfrak{q}}}{2\pi}-\bigg[n\frac{\theta_{\mathfrak{q}}}{2\pi}\bigg]\bigg\vert\leq \frac{1}{n},
$$
where $[\cdot]$ is either the floor or the ceiling part. Thus for each $n$ as above, we have 
$$
|c_{\bf f}(\mathfrak{a}_{n})|=|\delta| |\sin n\theta_{\mathfrak{p}}|<\frac{2\pi|\delta|}{n}.
$$
Now using the relation $\mathfrak{a}_{n}=\mathfrak{p}^{r_{1}}\mathfrak{q}^{n-1}$, we see that $n\geq c\log{\N}(\mathfrak{a}_{n})$, where $c>0$ is a constant depending on prime ideals $\mathfrak{p},\mathfrak{q}$ and number $r_{1}$. Since these prime ideals and numbers are fixed and depend only on the form $\bf f$. Thus we deduce that
$$
|c_{{\bf f}}(\mathfrak{a}_{n})|\leq\frac{\Lambda_{{\bf f}}}{\log{\N}(\mathfrak{a}_{n})}
$$
holds for infinitely many positive integers $n$ and hence for infinitely many integral ideals $\mathfrak{a}_{n}$. This completes proof of the theorem. \qed
\section{Proof of Theorem \ref{primepowerdiophantine}}
It follows from Lemma \ref{thetapnotrational} that there exist a set $\mathcal{S}$ of prime ideals with relative density $1$ such that for each prime ideal $\mathfrak{p}\in\mathcal{S}$, we have $\theta_{\mathfrak{p}}/(2\pi)$ is irrational.  
By Lemma \ref{Dirichlet}, there exists infinitely many positive integers $n$ such that 
$$
\bigg\Vert (n+1)\frac{\theta_{\mathfrak{p}}}{2\pi}\bigg\Vert=\bigg\vert (n+1)\frac{\theta_{\mathfrak{p}}}{2\pi}-\bigg[(n+1)\frac{\theta_{\mathfrak{p}}}{2\pi}\bigg]\bigg\vert\leq \frac{1}{(n+1)},
$$
where $[\cdot]$ is either the floor or the ceiling part. Thus for each $n$ as above, we have 
$$
|\sin (n+1)\theta_{\mathfrak{p}}|<\frac{2\pi}{(n+1)}.
$$
Hence by \eqref{cfpmexpression}, we have
$$
|c_{{\bf f}}(\mathfrak{p}^{n})|\leq\frac{\Lambda_{{\bf f}}}{(n+1)},
$$
where $\Lambda_{{\bf f}}=2\pi/ |\sin\theta_{\mathfrak{p}}|$. This proves the theorem. \qed

\section{Proof of Theorem \ref{Littlewoodtwoprime}}
It follows from Lemma \ref{thetapnotrational} that there exist a set $\mathcal{S}$ of prime ideals with relative density $1$ such that for each prime ideal $\mathfrak{p}\in\mathcal{S}$, we have $\theta_{\mathfrak{p}}/(2\pi)$ is irrational.. For any two distinct prime ideals $\mathfrak{p}$ and $\mathfrak{q}$ in $\mathcal{S}$, we have
$$
(n+1)|\sin (n+1)\theta_{\mathfrak{p}}|\leq(2\pi)(n+1)\bigg\Vert (n+1)\frac{\theta_{\mathfrak{p}}}{2\pi}\bigg\Vert
$$
and 
$$
|\sin (n+1)\theta_{\mathfrak{q}}|\leq(2\pi)\bigg\Vert (n+1)\frac{\theta_{\mathfrak{q}}}{2\pi}\bigg\Vert.
$$
Using \eqref{cfpmexpression}, we get from above,
$$
(n+1)|c_{{\bf f}}(\mathfrak{p}^{n})|\leq\frac{(2\pi)(n+1)}{|\sin \theta_{\mathfrak{p}}|}\bigg\Vert (n+1)\frac{\theta_{\mathfrak{p}}}{2\pi}\bigg\Vert
$$
and
$$
|c_{{\bf f}}(\mathfrak{q}^{n})|\leq\frac{2\pi}{|\sin \theta_{\mathfrak{q}}|}\bigg\Vert (n+1)\frac{\theta_{\mathfrak{q}}}{2\pi}\bigg\Vert.
$$
We deduce from above and the multiplicativity of $c_{{\bf f}}(\mathfrak{m})$ from \eqref{multiplicative} that
$$
(n+1)|c_{{\bf f}}(\mathfrak{p}^{n}\mathfrak{q}^{n})|\leq\frac{(4\pi)^{2}(n+1)}{|\sin \theta_{\mathfrak{p}}|\cdot|\sin \theta_{\mathfrak{q}}|}\bigg\Vert (n+1)\frac{\theta_{\mathfrak{p}}}{2\pi}\bigg\Vert\cdot\bigg\Vert (n+1)\frac{\theta_{\mathfrak{q}}}{2\pi}\bigg\Vert.
$$
Now the first part of Theorem \ref{Littlewoodtwoprime} follows from Conjecture \ref{Littlewood}. On the other hand, Conjecture \ref{Velani-Badziahin} concludes proof of the second part of Theorem \ref{Littlewoodtwoprime}.

\section{Proof of Theorem \ref{primepowerlowerthm}}
\begin{proof} Since $\theta_{\mathfrak{p}}/\pi$ is badly approximable, by definition, there exists a constant $\eta_{{\bf f}}>0$ such that 
\begin{equation}\label{thetapbadapprox}
\bigg\Vert (n+1)\frac{\theta_{\mathfrak{p}}}{\pi}\bigg\Vert >\frac{\eta_{{\bf f}}}{n+1}
\end{equation} 
for all positive integers $n$. Without loss of generality, we take $\eta_{{\bf f}}<1/4$. We write
$$
\gamma=(n+1)\frac{\theta_{\mathfrak{p}}}{\pi}-\bigg[(n+1)\frac{\theta_{\mathfrak{p}}}{\pi}\bigg]
$$
where $[\cdot]$ is either the floor or the ceiling part so that $-1/2\leq\gamma\leq1/2$. Therefore we have
\begin{equation}\label{sinpigamma}
\sin \pi\gamma=(-1)^{\big[(n+1)\frac{\theta_{\mathfrak{p}}}{\pi}\big]}\sin(n+1)\theta_{\mathfrak{p}}.
\end{equation}
Also from \eqref{thetapbadapprox}, we have
$$
\frac{\eta_{{\bf f}}\pi}{n+1}<|\pi\gamma|\leq\frac{\pi}{2}.
$$
Thus from the above, we have
$$
|\sin \pi\gamma|>\bigg\vert\sin\left(\frac{\eta_{{\bf f}}\pi}{n+1}\right)\bigg\vert.
$$
Using \eqref{sinpigamma} and above, we get
$$
|\sin(n+1)\theta_{\mathfrak{p}}|>\bigg\vert\sin\left(\frac{\eta_{{\bf f}}\pi}{n+1}\right)\bigg\vert.
$$
Since $|\sin x|\geq \frac{2}{\pi} |x|$ for $|x|\leq \frac{\pi}{2}$, we get from the above that
$$
|\sin(n+1)\theta_{\mathfrak{p}}|> 2\frac{\eta_{{\bf f}}}{n+1}.
$$ 
Finally using \eqref{cfpmexpression} in the above, we deduce
$$
|c_{{\bf f}}(\mathfrak{p}^{n})|>2\frac{\eta_{{\bf f}}}{n+1}.
$$
This completes proof of the theorem.
\end{proof}
\begin{remark}\label{rmkafterprf}	
\begin{enumerate}
 \item [{\bf (1).}] {\rm It would be an interesting question to investigate the Hecke eigen angles for which
$\mathfrak{\theta_{\mathfrak{p}}/\pi}$ is badly approximable. In fact, in \cite{Bengoechea}, It is raised (for $m=1$) whether the angles $\mathfrak{\theta_{\mathfrak{p}}/\pi}$ that are badly approximable are dense in the interval $[0,1]$. The distribution of such angles  have connections towards the refined rate of Diophantine approximation of real numbers by Fourier coefficients of cusp forms (see \cite{Bengoechea}). 
\item [{\bf (2).}] 
In \cite[Lemma 2, p. 393]{Murty-Murty-Shorey} and more generally in \cite[Proposition 2, p. 461]{Murty-Murty}, a lower bound is obtained for $|C_{f}(p^{n})|$ for any primitive form $f$ of weight $k\geq 4$ and level $N$ (when $m$=1) where $p$ is any prime and $n\geq 2$. In particular for $n$ even, it gives 
$$
|C_{f}(p^{n})|\geq p^{{\frac{k-1}{2}(n-c_{1}\log n)}}
$$
where $c_{1}>0$ is an effectively computable constant.
However if the angle $\mathfrak{\theta_{\mathfrak{p}}/\pi}$ is badly approximable then Theorem \ref{primepowerlowerthm} provides an improved bound for the Fourier coefficients $C_{{\bf f}}(\mathfrak{p}^{n})$ with $n\geq 1$ namely
$$
|C_{{\bf f}}(\mathfrak{p}^{n})|>\frac{2\eta_{{\bf f}}}{n+1}{\N}(\mathfrak{p}^{n})^{\frac{k_{0}-1}{2}}.
$$
\item [{\bf(3).}] Recently in \cite{Bhand-Gun-Rath}, the lower bound of $|C_{f}(p^{n})|$ have been generalized to Hilbert cusp forms and a bound for the height of $C_{{\bf f}}(\mathfrak{p}^{n})$ is obtained. More precisely, let $F\subset\Bbb R$ be  a finite Galois extension of $\Bbb Q$ of odd degree $m$. Let $\bf f$ be a primitive cusp form in $\mathscr{F}_{k}(\mathfrak{n})$ of parallel weight $k=(k,\ldots,k)$ such that $k\equiv 0\pmod{2}$.Then \cite[Theorem 2]{Bhand-Gun-Rath} will imply that for each $n\geq 6$ and for all but finitely many prime ideals $\mathfrak{p}$, we have 
$$
H(C_{{\bf f}}(\mathfrak{p}^{n}))\gg {\N}(\mathfrak{p})^{\delta_{1}}
$$
where $H(\cdot)$ is as defined in subsection \ref{heightfunction} and $\delta_{1}$ is an effectively computable constant which depends on the number field $\Bbb Q(\bf f)$ which is generated over $\Bbb Q$ by $C_{{\bf f}}(\mathfrak{m})$, for all $\mathfrak{m}$ (see \cite[Proposition 2.8, p. 654]{Shimura}) but independent of $\mathfrak{p}$ and $n$. 

However, it is possible to obtain a bound on the height of $C_{{\bf f}}(\mathfrak{p}^{n})$ uniformly for each $n\geq 1$ without any restriction on the number field $F$ and weight $k$ of the form $\bf f$ provided the angle $\theta_{\mathfrak{p}}/\pi$ is badly approximable. In this case, it follows from Theorem \ref{primepowerlowerthm} that 
$$
H(C_{{\bf f}}(\mathfrak{p}^{n}))\geq \left(2\frac{\eta_{{\bf f}}}{n+1}\right)^{1/d} {\N}(\mathfrak{p}^{n})^{\frac{k_{0}-1}{2d}}
$$
where $d:=[\Bbb Q({\bf f}):\Bbb Q]$.
Hence we deduce from the above that
$$
H(C_{{\bf f}}(\mathfrak{p}^{n}))\gg {\N}(\mathfrak{p}^{n})^{\frac{k_{0}-1-\epsilon}{2d}}\geq{\N}(\mathfrak{p})^{\frac{k_{0}-1-\epsilon}{2d}}.
$$
\item [{\bf(4).}] When $m=1$, a deep conjecture of Atkin and Serre \cite{Atkin-Serre} asserts that for $k\geq 4$ and all $\epsilon>0$, there exist constants $\gamma_{\epsilon, f}>0$ and $\gamma^{'}_{\epsilon, f}>0$ (Here ${\bf f}:=f$ is an elliptic cusp form) such that if $p>\gamma^{'}_{\epsilon, f}$, then 
\begin{equation*}
|C_{f}(p)|\geq \gamma_{\epsilon, f}\ p^{\frac{k-3}{2}-\epsilon}.
\end{equation*}
However, Theorem \ref{primepowerlowerthm} provides an improved lower bound for the coefficients $C_{{\bf f}}(\mathfrak{p})$ under certain assumptions. More precisely,  if the angle $\theta_{\mathfrak{p}}/\pi$ is badly approximble, then we have
\begin{equation*}
|C_{{\bf f}}(\mathfrak{p})|\geq \eta_{{\bf f}}{\N}(\mathfrak{p})^{\frac{k_{0}-1}{2}}.
\end{equation*}
}
\end{enumerate}
\end{remark}
\section*{Acknowledgements}
The author sincerely thanks Professor Florian Luca and Professor Igor E. Shparlinski for the helpful comments over an email conversation.

\end{document}